\documentclass{birkjour}
\usepackage{url}
\usepackage{float}
\usepackage{tikz-cd}
\usepackage{mathtools}
\usepackage{dsfont}

\usepackage{hyperref}

\usepackage[capitalize,nameinlink]{cleveref}
\newcommand{\crefdetail}[2]{\hyperref[#1]{\namecref{#1}~\labelcref*{#1}~\ref*{#2}}}

\usepackage{stmaryrd}

\DeclarePairedDelimiter\autobracket{(}{)}
\newcommand{\bb}[1]{\autobracket*{#1}}
\DeclarePairedDelimiter\autobracketcurly{\{}{\}}
\newcommand{\bc}[1]{\autobracketcurly*{#1}}
\DeclarePairedDelimiter\autobracketsquare{[}{]}
\newcommand{\bs}[1]{\autobracketsquare*{#1}}

\DeclarePairedDelimiter\abs{\lvert}{\rvert}
\DeclarePairedDelimiter\norm{\lVert}{\rVert}

\makeatletter
\let\oldabs\abs
\def\abs{\@ifstar{\oldabs}{\oldabs*}}
\let\oldnorm\norm
\def\norm{\@ifstar{\oldnorm}{\oldnorm*}}
\makeatother

\makeatletter
\let\oldtilde\tilde
\def\tilde{\@ifstar{\widetilde}{\oldtilde}}
\makeatother

\let\oldforall\forall
\renewcommand{\forall}{\oldforall \, }
\let\oldexist\exists
\renewcommand{\exists}{\oldexist \: }

\newcommand{\defi}{\coloneqq}

\newcommand{\N}{\mathbb{N}} % natural numbers
\newcommand{\R}{\mathbb{R}} % real numbers
\newcommand{\inv}[1]{{#1}^{-1}}
\newcommand{\id}{\text{id}}

\makeatletter
\newcommand{\restrscale}[2]{\left.#1\right|_{#2}}
\newcommand{\restrnoscale}[2]{{#1} \vert_{#2}}
\def\restr{\@ifstar{\restrnoscale}{\restrscale}}
\makeatother

\DeclarePairedDelimiter\floor{\lfloor}{\rfloor}

\newcommand{\1}{\mathds{1}}

 \newtheorem{thm}{Theorem}[section]
 \newtheorem{cor}[thm]{Corollary}
 \newtheorem{lem}[thm]{Lemma}
 
 \theoremstyle{definition}
 \newtheorem{defn}[thm]{Definition}
 \theoremstyle{remark}
 
 \newtheorem{ex}[thm]{Example}
 \numberwithin{equation}{section}

\newcommand{\n}[1]{{\left\|{#1}\right\|}}

\begin{document}

%-------------------------------------------------------------------------
% editorial commands: to be inserted by the editorial office
%
%\firstpage{1} \volume{228} \Copyrightyear{2004} \DOI{003-0001}
%\seriesextraline{This is the Concrete Title of this Book\br H.E. R and S.T.C. W, Eds.}
%\firstpage{1}

\title[Integral RB Operator]
 {An Integral RB Operator}

%----------Author 1
\author[Jahn]{Marvin Jahn}

\address{%
Department of Mathematics\\
Technical University of Munich\\
Boltzmannstr. 3\\
85748 Garching b. München\\
Germany}

\email{marvin.jahn@tum.de}

%\thanks{This work was completed with the support of our \TeX-pert.}
%----------Author 2
\author[Massopust]{Peter Massopust}
\address{Department of Mathematics\\
Technical University of Munich\\
Boltzmannstr. 3\\
85748 Garching b. München\\
Germany}

\email{massopust@ma.tum.de}


%\date{January 1, 2004}

\begin{abstract}
We introduce the novel concept of integral Read-Bajraktarevi\'c (iRB) operator and discuss some of its properties. We show that this iRB operator generalizes the known Read-Bajraktarevi\'c (RB) operator and we derive conditions for the fixed point of the iRB operator to belong to certain function spaces.
\end{abstract}

%%% ----------------------------------------------------------------------
\maketitle
%%% ----------------------------------------------------------------------
%\tableofcontents
\section{Introduction}
The idea of interpolating a given set of data points by well-behaved functions is a prominent and very important concept in many areas of mathematics and the natural sciences. Typical choices for such functions are, for instance, polynomials or splines, but one may also wish to interpolate the points using less smooth or even non-differentiable functions. This is the case when the data come from an underlying model that describes highly irregular and complex signals or images.

This desire to approximate a data set by a non-smooth function leads to the concept of \textit{fractal interpolation} originally introduced in \cite{barnsley1986} and in a slightly different setting in \cite{H}. Since then, numerous variants of this concept have been studied. The interested reader is referred to the following albeit incomplete list of references regarding fractal interpolation and some of its applications: \cite{BHVV,bhm,bedford,DLM,dubuc2,KCM,LDV,massopust1,m5,N,SB,SB1}. More pointers to the flexibility and versatility of fractal interpolation theory are found in the  literature part of these papers.

For some of these variations, the focus has shifted from trying to interpolate certain points to the general construction of fractal functions given some parameters. This is also the setup considered in \cite{massopust2022}, an approach which we seek to generalize in this paper. 

In \cite{massopust2022}, a normed vector space $E$, a Banach space $F$ and a subset $X \subset E$ are fixed and one is interested in finding a function $\psi \colon X \to F$ satisfying $n$ functional equations of the form
\begin{equation}
  \label{eq:1}
  \psi(l_i(x)) = q_i(x) + s_i(x) \psi(x) \qquad \forall x \in X, i \in\N_n,
\end{equation}
where $l_i$, $q_i$, $s_i$ are certain fixed functions. Here and throughout the paper, we use the notation $\N_n$ to denote the initial segment of length $n$ of the natural numbers $\N := \{1, 2, \ldots\}$.

One of the assumptions in \cite{massopust2022} is that the $n$ functions $l_i \colon X \to X$ are injective and that their images $X_i \defi l_i(X)$ partition $X$. (The case where some of the $X_i$ have points in common is treated in \cite{SB,SB1}.) 

The functional equations \eqref{eq:1} express a kind of self-referentiability that their solutions must exhibit. (Cf. for instance, \cite{barnsley1986,massopust1}.) Thus, these solutions are also called \textit{fractal (interpolation) functions} as the graph of $\psi$ is in general a fractal (set), i.e., a set with in general non-integral (box or Hausdorff) dimension.

A key realization is that finding a solution of \eqref{eq:1} is equivalent to determining a fixed point of the \textit{Read-Bajraktarevi\'c} (\textit{RB}) operator \cite{bedford,SB,SB1} $T:Z\to Z$,
\begin{align}
  \label{eq:2}
  T&(f)(x) := \nonumber\\
  & \sum_{i=1}^n (q_i \circ \inv{l_i})(x) \1_{X_i}(x) + \sum_{i=1}^n (s_i \circ \inv{l_i})(x) \cdot (f \circ \inv{l_i})(x) \1_{X_i}(x),
\end{align}
where $Z$ denotes a suitable subset of the vector space $V(X,F)$ of all functions $X \to F$ and $\1_S$ the indicator or characteristic function of a set $S$.

The key insight is to choose $Z$ in such a way that it can be endowed with a metric (usually even a norm) such that $Z$ becomes a complete metric space.
It suffices then to show that $T$ is a contraction on $Z$ in order to establish the existence and uniqueness of a fixed point
of $T$ by the Banach fixed point theorem.

The novel idea presented in this paper is to replace the sums appearing in the RB operator \eqref{eq:2} by integrals (with respect to Lebesgue measure). This leads us to the definition of an \textit{integral Read-Bajraktarevi\'c} (\textit{iRB}) operator,
a new concept that is formulated in \cref{sec:integral-RB-operator}. In the subsequent Section 3, we introduce a method to extend a discrete collection $\{\tilde{l}_i : i \in\N_n\}$ of maps $\tilde{l}_i \colon X \to Y$ to a function $l \colon [1,n] \times X \to Y$ which may be thought of as a family of functions $l_t \colon X \to Y$ for $t \in [1,n]$ by a suitable homotopy.
This construction serves two purposes. On the one hand, we use it to specify in \cref{sec:irb-operator-generalizes-rb} in what sense the iRB operator generalizes the RB operator \eqref{eq:2}, and on the other hand, it allows to present some concrete examples.

In \cref{sec:bounded-solutions}, we consider the case that $Z = B(X,F)$, the subspace of $V(X,F)$ consisting of bounded functions equipped with the supremum norm, and explore under what assumptions on the parameter functions $l$, $s$, and $q$, we can guarantee that the iRB operator $T$ is a contraction on $Z$ thus yielding a unique fixed point in $B(X,F)$ by the Banach fixed point theorem.

In the subsequent \cref{sec:irb-operator-generalizes-rb}, we exhibit the RB operator \eqref{eq:2} as a special case of the iRB operator. In the following section, we investigate under what conditions the bounded fixed point is continuous. Finally, in \cref{sec:solutions-in-lp}, we turn our attention to the space of $L^p$-functions, $Z = L^p$, endowed with the usual $p$-norm and investigate under what assumptions the fixed point of the iRB operator $T$ exists and is unique.
\section{An integral RB Operator}\label{sec:integral-RB-operator}
Throughout this paper, let $E$ denote a normed vector space, $F$ a Banach space, and $X \subset E$ a nonempty subset of $E$. Moreover, let $Z$ denote a fixed subset of the vector space $V(X,F)$ of all functions $X \to F$, to be specified later.
Our first goal is to derive an ``integral'' version of the RB operator \eqref{eq:2} described in the introduction.

To this end, we observe that the sums within the RB operator \eqref{eq:2}
are just integrals with respect to counting measure $\mu$.
Indeed, for a fixed $x \in X$, we may define $a_i \defi (q_i \circ \inv{l_i})(x) \1_{X_i}(x)$ in order to rewrite the first sum in \eqref{eq:2} as
\begin{align*}
\sum_{i=1}^n a_i \cdot \mu(\bc{i}) = \sum_{i=1}^n \int_{\N} \1_{\bc{i}}(k) a_k d \mu(k) =  \int_\N \1_{\N_n}(k) a_k d \mu(k)
= \int_{\N_n} a_k d \mu(k) 
\end{align*}
and similarly for the second sum. Thus, it seems natural to replace the sums by integrals with respect to Lebesgue measure:
\begin{align*}
T \colon Z &\to Z,\\ 
T(f)(x) &= \int_1^n (q_t \circ \inv{l_t})(x) \1_{X_t}(x) d t \\
& \qquad + \int_1^n (s_t \circ \inv{l_t})(x) \cdot (f \circ \inv{l_t})(x) \1_{X_t}(x) d t. 
\end{align*}
We call this new mapping an \textit{integral Read-Bajraktarevi\'c} ({iRB}) operator.
Note that just like the RB operator \eqref{eq:2}, the iRB operator consists of a linear part (the second term) and an affine part (the first term), i.e., it is also an affine operator.

Instead of $n$ functional equations \eqref{eq:1}, a fixed point $\psi \colon X \to F$ of $T$ must now satisfy the integral equation
\[ \psi(x) = \int_1^n (q_t \circ \inv{l_t})(x) \1_{X_t}(x) d t + \int_1^n (s_t \circ \inv{l_t})(x) \cdot (\psi \circ \inv{l_t})(x) \1_{X_t}(x) d t. \]
Therefore, we also call the fixed points of $T$ \textit{solutions}.

Of course, the step of replacing sums by integrals requires modifying our setup.
Instead of $n$ injective functions $l_i \colon X \to X$ for $i \in \N_n$, we consider
a family of measurable functions $l \colon [1,n] \times X \to X$ such that the $l_t \defi l(t,\cdot) \colon X \to X$ are injective for almost all $t \in [1,n]$.

Following the above notation, we write for the image of each of the $l_t$ again $X_t \defi l_t(X)$. However, we do not assume that the $X_t$ (for $t \in [1,n]$) are disjoint or cover $X$. Note that $T(f)(x)$ is zero whenever $x \notin \bigcup\limits_{t \in [1,n]} X_t$ which means that those $x$ are not of particular interest.

In a similar way, we proceed with the $q_i$ and $s_i$. Instead of $n$ such functions we now have two measurable functions $q \colon [1,n] \times X \to F$ and $s \colon [1,n] \times X \to \R$.

For $x \in X$, we call the sets
\[ 
T_x \defi \bc{t \in [1,n] : x \in X_t} = \bc{t \in [1,n] : \exists y \in X: l(t,y) = x} \subset [1,n] 
\]
\textit{$x$-hit times}.

Thus, we may rewrite the iRB operator as follows:
\begin{equation}
  \label{eq:4}
  T(f)(x) = \int_{T_x} (q_t \circ \inv{l_t})(x) d t + \int_{T_x} (s_t \circ \inv{l_t})(x) \cdot (f \circ \inv{l_t})(x) d t.
\end{equation}
Therefore, in order for the iRB operator to be well-defined, we have to require that $T_x$ is measurable for every $x \in X$. Otherwise, $l_t$ could be, e.g., the identity function on the Vitali set for every $t \in [1,n]$ so that $T_x$ is the Vitali set for every $x \in X$ and thus not measurable.

If $X$ is compact, then $T_x$ is closed and therefore measurable, as the following lemma shows.
\begin{lem}
  Suppose that $X \subset E$ is compact and that $l \colon [1,n] \times X \to X$ is continuous.
  Then the $x$-hit times $T_x$ are closed for every $x \in X$.
\end{lem}
\begin{proof}
  Let $t_k \to t$ be a convergent sequence in $[1,n]$, such that for every $k \in \N$,
  there exists an $x_k \in X$ with $l(t_k,x_k) = x$.
  Since $X$ is compact, there exists a subsequence of $(x_k)_{k \in \N}$
  converging to some $y \in X$.
  Replacing the original sequence with this subsequence,
  we have $(t_k,x_k) \to (t,y)$, so that the continuity of $l$ yields $t \in T_x$,
  which had to be shown.
\end{proof}

It turns out that $T_x$ plays an important role in characterizing the behavior of the iRB operator. (See \cref{thm:fixed-point-bounded} below.)
Note that in the discrete setting, it is common to assume that the $X_i$ are pairwise disjoint for $i \in\N_n$ so that the discrete analogue $\bc{i \in\N_n : x \in X_i}$ of the $x$-hit times $T_x$ is just a singleton for every $x \in X$.
\section{Construction of an Extension}
Having introduced the iRB operator, we now outline a method to construct a map $l \colon [1,n] \times X \to Y$ given $n$ functions $\tilde{l}_i \colon X \to Y$.
This new function $l$ is then an extension of the $\tilde{l}_i$ in the sense that
$l(i,\cdot) = \tilde{l}_i$ for all $i \in\N_n$.
Our interest in this construction is motivated by the fact that it will allow us to understand the iRB operator
as a generalization of the RB operator \eqref{eq:2}.
Furthermore, it will enable us to easily provide some examples for the solutions that may arise for typical choices of
the parameter functions $l$, $q$ and $s$.

The construction relies on choosing a particular function $h \colon [0,1] \to [0,1]$,
which we often just take to be the identity function $\id_{[0,1]}$.

\begin{defn}\label{defn:extension}
  Let $X \subset X'$ and $Y \subset Y'$ be subsets of the normed vector spaces $X'$ and $Y'$, respectively, with $Y$ being convex.
  Furthermore, let $h \colon [0,1] \to [0,1]$ be a function with $h(0) = 0$ and $h(1) = 1$.
  
  For a collection of functions $\tilde{l}_i \colon X \to Y$ ($i \in\N_n$), define their {extension} with respect to $h$ to be the function
  \begin{gather*}
  l \colon [1,n] \times X \to Y,\\
   l(t,\cdot) \defi (1 - h(t - \floor{t})) \cdot \tilde{l}_{\floor{t}} + h(t - \floor{t}) \cdot \tilde{l}_{\floor{t}+1}.  
 \end{gather*}
\end{defn}
\noindent
Here and below, $\lfloor\cdot\rfloor$ denotes the floor function. 

Intuitively, the function $h$ distributes the ``influence'' that $\tilde{l}_{\floor{t}}$ and $\tilde{l}_{\floor{t}+1}$
have on $l_t$. If $h(t - \floor{t})$ is small, then $l_t$ ``is close to'' $\tilde{l}_{\floor{t}}$ and
if $h(t - \floor{t})$ is close to $1$, then $l_t$ approximates $\tilde{l}_{\floor{t}+1}$.

The extension inherits some properties from the $\tilde{l}_i$, which we summarize in the following lemma.
\begin{lem}\label{lem:extension-properties}
  The extension $l$ given in \cref{defn:extension} has the following properties:
  \begin{enumerate}
  \item If $h$ and all the $\tilde{l}_i$ are continuous, then $l$ is continuous.
  \item If all $\tilde{l}_i$ are bounded by some positive constant, then $l$ is bounded by that same constant.
  \end{enumerate}
\end{lem}
\begin{proof}
  \begin{enumerate}
  \item Let $t_k \to t$ be a sequence in the interval $[1,n]$ and $x_k \to x$ a sequence in $X$.
    If $t\notin\N$  then $\floor{t_k}$ is constant for sufficiently large $k$ and the claim
    follows by continuity of $h$ and all the $\tilde{l}_i$.
    
    In case that $t\in\N$, we split the sequence $(t_k,x_k)$ into two subsequences with $t_k \ge t$ and $t_k < t$, respectively. Then we show that both subsequences converge to the same limit by considering the following two cases:
    \begin{description}
    \item[Case 1: $t_k \ge t,\ \forall k \in \N$] Then, for sufficiently large $k$, we have
      \[ l(t_k,x_k) = (1 - h(t_k - t)) \cdot \tilde{l}_{t}(x_k) + h(t_k - t) \cdot \tilde{l}_{t+1}(x_k), \]
      so $l(t_k,x_k) \to \tilde{l}_t(x)$ since $h(0) = 0$.
    \item[Case 2: $t_k < t,\ \forall k \in \N$] Similarly, for sufficiently large $k$, we observe that
      \[ l(t_k,x_k) = (1 - h(t_k - (t - 1))) \cdot \tilde{l}_{t-1}(x_k) + h(t_k - (t - 1)) \cdot \tilde{l}_t(x_k), \]
      showing that $l(t_k,x_k) \to \tilde{l}_t(x)$ as $h(1) = 1$.
    \end{description}
  \item This is immediate, because on each interval $[i,i+1]$ ($i \in\N_{n-1}$), $l$ is a convex combination of $\tilde{l}_i$
    and $\tilde{l}_{i+1}$. \qedhere
  \end{enumerate}
\end{proof}

\cref{lem:extension-properties} shows that if $h$ and all $\tilde{l}_i$ are continuous, then $l$ amounts to a
homotopy between $\tilde{l}_i$ and $\tilde{l}_{i+1}$ for every $i \in\N_{n-1}$:
\begin{center}
  \begin{tikzcd}[column sep = large]
    [0,1] \times X \ar[r,"(\cdot + i) \times \id_X"] & {[i,i+1] \times X} \ar[r,hook] & {[1,n] \times X} \ar[r,"l"] & Y.
  \end{tikzcd}
\end{center}

Note that the $l_t \colon X \to Y$ ($t \in\N_n$) resulting from our construction of the extension $l$ need not be injective, even if all $\tilde{l}_i$ ($i \in\N_n$) are injective. We now give a simple example demonstrating this phenomenon.

\begin{ex}\label{exa:0}
  Let $X = Y = [0,1] \subset \R$ and $h \colon [0,1] \to [0,1]$ a continuous function with $h(0) = 0$ and $h(1) = 1$.
  The affine linear functions
  \[ \tilde{l}_1 \colon [0,1] \to [0,1],\ x \mapsto \tfrac{1}{2} x, \qquad \tilde{l}_2 \colon [0,1] \to [0,1],\ x \mapsto 1 - \tfrac{1}{2} x  \]
  are injective and have the extension
  \begin{align*}
  l(t,x) &= \bb{1 - h(t - \floor{t})} \cdot \tfrac{1}{2} x + h(t - \floor{t}) \cdot \bb{1 - \tfrac{1}{2} x}\\
    &= \bb{\tfrac{1}{2} - h(t - \floor{t})} \cdot x + h(t - \floor{t}). 
    \end{align*}
  By continuity of $h$, the preimage $\inv{h}(\frac{1}{2})$ is nonempty and any element $u$ in this preimage
  has the property that $l(u+1,\cdot)$ is constant.
  For instance, by choosing $h := \id_{[0,1]}$, we have $l(\frac{3}{2},\cdot) = \frac{1}{2}$.
  
  Similarly, replacing $\tilde{l}_2$ by the injective function
  \[ 
    \tilde{l}_2 \colon [0,1] \to [0,1],\ x \mapsto 1 - \tfrac{1}{2} x^2, 
  \]
  yields the extension (for $h = \id_{[0,1]}$)
  \[ l(t,x) = - \tfrac{1}{2} (t-1) \cdot x^2 + \tfrac{1}{2} (2-t) \cdot x + t - 1. \]
  In this case, the quadratic polynomial $l_t$ is not injective if and only if its maximum lies in the interval $(0,1)$,
  which turns out to be the case for $t \in (\frac{1}{3},1)$.
  This shows that it need not be the case that $l_t$ is injective for almost all $t$,
  even though all $\tilde{l}_i$ are injective.
\end{ex}

\section{Bounded Solutions}\label{sec:bounded-solutions}
We now seek bounded solutions for the iRB operator, that is, we ask under what conditions the iRB operator $T$ maps the Banach space $Z \defi B(X,F)$
of bounded functions $X \to F$ to itself and when it is a contraction.

Here, $Z$ is equipped with the supremum norm $\norm{f}_\infty \defi \sup\limits_{x \in X} \norm{f(x)}$.
Note that $T(f)$ is certainly bounded if $f \colon X \to F$, $q \colon [1,n] \times X \to F$ and $s \colon [1,n] \times X \to \R$
are bounded.

\begin{thm}\label{thm:fixed-point-bounded}
  Suppose that $q \colon [1,n] \times X \to F$ and $s \colon [1,n] \times X \to \R$ are bounded.
  Let $S \defi \sup\limits_{t \in [1,n], x \in X} \abs{s(t,x)}$ be the supremum of $s$ and
  \[ M \defi \sup_{x \in X} \lambda(T_x) = \sup_{x \in X} \lambda(\bc{t \in [1,n]: x \in X_t}) \]
  denotes the maximal Lebesgue measure $\lambda$ of all hit times $T_x$.
  If $S \cdot M < 1$, then the iRB operator $T \colon B(X,F) \to B(X,F)$ possesses a unique fixed point.
\end{thm}
\begin{proof}
  As $B(X,F)$ endowed with the supremum norm is a Banach space, it suffices to show that $T$ is a contraction, since the result is then a consequence of the Banach fixed point theorem.
  
  Indeed, for $f,g \in B(X,F)$ and $x \in X$, we obtain (using \eqref{eq:4})
  \begin{align*}
    \norm{T(f)(x) - T(g)(x)}_\infty
    &= \norm{\int_{T_x} (s_t \circ \inv{l_t})(x) \cdot ((f-g) \circ \inv{l_t})(x) d t}_\infty\\
    &\le \int_{T_x} \abs{(s_t \circ \inv{l_t})(x)} \cdot \norm{((f-g) \circ \inv{l_t})(x)}_\infty d t\\
    &\le S \cdot M \cdot \norm{f-g}_\infty. \qedhere
  \end{align*}
\end{proof}

By virtue of the constructive nature of the proof of the Banach fixed point theorem, we obtain immediately the following corollary. 
\begin{cor}\label{cor:fixed-point-bounded-speed-conv}
  The fixed point $f \in B(X,F)$ of the iRB operator $T$ from \cref{thm:fixed-point-bounded} is the limit of the iterates
  $f_k \defi T(f_{k-1})$, for $k \in\N$, and an arbitrarily chosen $f_0 \in B(X,F)$.
  Furthermore, we have the estimate
  \[
  \norm{f_k - f}_\infty \le \frac{S \cdot M}{1 - S \cdot M} \cdot \norm{f_k - f_{k-1}}_\infty,\quad k\in \N.
  \]
\end{cor}
\noindent
We illustrate the above theorem by hand of two examples.

\begin{ex}\label{exa:1}
  Let $X = [0,1] \subset \R$ and $F = \R$. Consider the affine linear functions
  \[ \tilde{l}_1 \colon [0,1] \to [0,1],\ x \mapsto \tfrac{1}{2} x, \qquad \tilde{l}_2 \colon [0,1] \to [0,1],\ x \mapsto \tfrac{1}{2} x + \tfrac{1}{2}. \]
  Their extension for $h \colon [0,1] \to [0,1]$ is
  \begin{equation}
    \label{eq:3}
    l(t,x) = \bb{1 - h(t - \floor{t})} \cdot \tfrac{1}{2} x + h(t - \floor{t}) \cdot \bb{\tfrac{1}{2} x + \tfrac{1}{2}}
    = \tfrac{1}{2} \bb{x + h(t - \floor{t})}.
  \end{equation}
  For $t \in [1,2]$ the inverse of $l_t$ is thus given by $\inv{l_t}(x) = 2x - h(t - \floor{t})$.
  % By \cref{lem:extension-properties}, $l_t \colon [0,1] \to [0,1]$ is strictly monotonically increasing
  % for every $t \in [1,2]$
  The $x$-hit times $T_x = \bc{t \in [1,2] : x \in X_t}$ trivially satisfy $M \le 1$ in \cref{thm:fixed-point-bounded}.
  Furthermore, we choose $h := \id_{[0,1]}$ and set
  \begin{align*}
    q \colon [1,2] \times [0,1] \to \R,&\quad (t,x) \mapsto
    \begin{cases}
      1, & x \ge 2 - t\\
      0, & x < 2 - t
    \end{cases},\\
    s \colon [1,2] \times [0,1] \to \R,&\quad (t,x) \mapsto \tfrac{1}{2} \cdot x \cdot (t - 1).
  \end{align*}
  Since $S \cdot M \le S = \frac{1}{2} < 1$, the theorem yields the existence of a unique fixed point of the iRB operator,
  which is plotted in the figure below.
  \begin{figure}[H]
    \centering
    \includegraphics[width=3\textwidth/4]{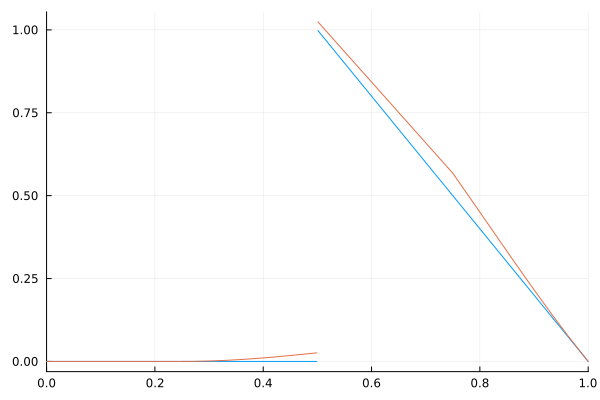}
    \vspace{-5pt}
    \caption[]{The graphs of the first (blue) and second (orange) iterate of the iRB operator,
      when starting with the zero function. The third iterate $f_3$ (not plotted) is almost indistinguishable
      from the second one $f_2$. The estimate $\norm{f_3 - f}_\infty \le \norm{f_3 - f_2}_\infty$
      from \cref{cor:fixed-point-bounded-speed-conv} for the fixed point $f$ implies that
      $f$ ``essentially looks like'' $f_3$ (at least at this scale).}
    \label{fig:exa-disc}
  \end{figure}
\end{ex}

\begin{ex}\label{exa:2}
  Choosing the same $l$ and $s$ as in \cref{exa:1} and replacing $q$ by
  \[ q \colon [1,2] \times [0,1] \to \R,\ (t,x) \mapsto 2 \cdot x \cdot (t - 1) \]
  yields the following graph:
  \begin{figure}[H]
    \centering
    \includegraphics[width=3\textwidth/4]{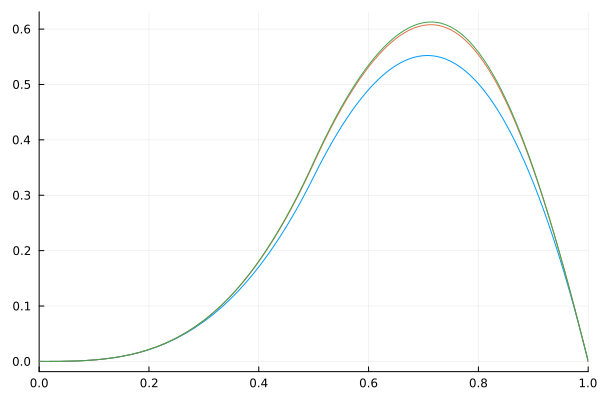}
    \vspace{-5pt}
    \caption{The graphs of the first (blue), second (orange) and third (green) iterate of the iRB operator,
      when starting with the zero function.}
    \label{fig:exa-cont}
  \end{figure}
\end{ex}

These two examples show that depending on the choice of the parameter functions $s$ and $q$,
the solutions can have a smooth or discontinuous character.
In \cref{sec:continuous-sol}, we derive a criterion guaranteeing continuity of the solution.
Furthermore, we will see in the next section that any solution of an RB operator can also be realized as a
solution of a corresponding iRB operator.
Therefore, the solutions of iRB operators can display the same ``richness'' in behavior as those of RB operators.
\section{The iRB Operator as a Generalization of the RB Operator}\label{sec:irb-operator-generalizes-rb}
The iRB operator can be seen to be a generalization of the RB operator \eqref{eq:2} using the construction of an extension presented in  \cref{defn:extension}.

To this end, let $\tilde{l}_i$, $\tilde{q}_i$ and $\tilde{s}_i$ be given parameter functions of the RB operator \eqref{eq:2},
which we want to exhibit as a special case of the iRB operator $T$ \eqref{eq:4}.
For $h := \1_{[\frac{1}{2},1]} \colon [0,1] \to [0,1]$, we denote the extensions of $\tilde{l}_i$, $\tilde{q}_i$ and $\tilde{s}_i$
by $l$, $q$, and $s$, respectively.
This means that
\[
  l_t =
  \begin{cases}
    \tilde{l}_{\floor{t}}, & t - \floor{t} < \frac{1}{2};\\
    \tilde{l}_{\floor{t}+1}, & t - \floor{t} \ge \frac{1}{2},
  \end{cases}
\]
and similarly for $q_t$ and $s_t$.
In other words, for $i \in \N_{n-1}$ and $t \in [i - \frac{1}{2},i + \frac{1}{2})$, $l_t$ is equal to $\tilde{l}_i$.
Furthermore, $l_t$ is equal to $\tilde{l}_1$ if $t \in [1,\frac{3}{2})$ and to $\tilde{l}_n$ for $t \in [n - \frac{1}{2},n]$.
Notice that the situation is slightly asymmetrical because $l$
agrees with the $\tilde{l}_i$ for $i \in \bc{2,\dots,n-1}$ on an interval of length $1$ (after projecting onto the first component),
whereas $l$ is equal to $\tilde{l}_1$ or $\tilde{l}_n$ only on an interval of length $\frac{1}{2}$.

To remedy this, we need to replace $\tilde{q}_1$ by $2 \cdot \tilde{q}_1$ and do the same for $\tilde{q}_n$, $\tilde{s}_1$, and $\tilde{s}_n$.
To simplify notation, we also set
\begin{align*}
  I(i,z,z')(x)
  &\defi \int_z^{z'} (q_i \circ \inv{l_i})(x) \1_{X_i}(x) d t \\
  & \qquad + \int_z^{z'} (s_i \circ \inv{l_i})(x) \cdot (f \circ \inv{l_i})(x) \1_{X_i}(x) d t\\
  &= (z' - z) \cdot (q_i \circ \inv{l_i})(x) \1_{X_i}(x)\\
  & \qquad + (z' - z) \cdot (s_i \circ \inv{l_i})(x) \cdot (f \circ \inv{l_i})(x) \1_{X_i}(x),
\end{align*}
for $i \in\N_n$, $z, z' \in [1,n]$, and $x \in X$.

Then, the iRB operator induced by $l$, $q$ and $s$ agrees with the ordinary RB operator \eqref{eq:2} for $\tilde{l}_i$, $\tilde{q}_i$, and $\tilde{s}_i$:
\begin{align*}
  T(f)(x)
  &= \int_1^n (q_t \circ \inv{l_t})(x) \1_{X_t}(x) d t + \int_1^n (s_t \circ \inv{l_t})(x) \cdot (f \circ \inv{l_t})(x) \1_{X_t}(x) d t\\
  &= I\bb{1,1,\frac{3}{2}}(x) + \sum_{i=2}^{n-1} I\bb{i,i - \frac{1}{2},i + \frac{1}{2}}(x) + I\bb{n,n - \frac{1}{2},n}(x)\\
  &= \sum_{i=1}^n (\tilde{q}_i \circ \inv{\tilde{l}_i})(x) \1_{\tilde{X}_i}(x) + \sum_{i=1}^n (\tilde{s}_i \circ \inv{\tilde{l}_i})(x) \cdot (f \circ \inv{\tilde{l}_i})(x) \1_{\tilde{X}_i}(x).
\end{align*}
\noindent
We summarize our result in the next theorem.
\begin{thm}[\textbf{iRB operator generalizes the RB operator}]\label{thm:iRB-gen-RB}~\\
  With $l$, $q$ and $s$ constructed as above, their induced iRB operator is equal to
  the RB operator of $\tilde{l}_i$, $\tilde{q}_i$ and $\tilde{s}_i$.
\end{thm}

In particular, any function that is a solution of the RB operator can also be realized as a solution of the corresponding iRB operator.

\begin{ex}
\label{exa:3}
  For example, the RB operator \eqref{eq:2} induced by
  \begin{align*}
    \tilde{l}_1 \colon [0,1) \to [0,1),\ x \mapsto \tfrac{1}{2} x, &\qquad \tilde{l}_2 \colon [0,1) \to [0,1),\ x \mapsto \tfrac{1}{2} x + \tfrac{1}{2},\\
    \tilde{q}_1 \colon [0,1) \to \R,\ x \mapsto \tfrac{1}{2} x,  &\qquad \tilde{q}_2 \colon [0,1) \to \R,\ x \mapsto - \tfrac{1}{2} x + \tfrac{1}{2},\\
    \tilde{s}_1 \colon [0,1) \to \R,\ x \mapsto \tfrac{1}{4},  &\qquad \tilde{s}_2 \colon [0,1) \to \R,\ x \mapsto \tfrac{1}{4}
  \end{align*}
  is
  \[
    \tilde{T}(f)(x) =
    \begin{cases}
      x + \frac{1}{4} \cdot f(2x), & x < \frac{1}{2};\\
      -x + 1 + \frac{1}{4} \cdot f(2x-1), & x \ge \frac{1}{2}.
    \end{cases}
  \]
  A straightforward calculation shows that the parabola $x \mapsto 2x(1-x)$ is a fixed point of $\tilde{T}$.
  We replace $\tilde{q}_1$ by $x \mapsto x$, $\tilde{q}_2$ by $x \mapsto -x + 1$,
  $\tilde{s}_1$ by $x \mapsto \frac{1}{2}$ and $\tilde{s}_2$ by $x \mapsto \frac{1}{2}$
  and then form the extensions $l$, $s$ and $q$ with respect to $h=\1_{[\frac{1}{2},1]}$.
  By \cref{thm:iRB-gen-RB}, the resulting iRB operator $T$ is precisely the RB operator $\tilde{T}$.
  As the iRB operator has a unique fixed point by \cref{thm:fixed-point-bounded},
  the iterates $(T^k(0))_{k \in \N}$ (when, e.g., starting with the zero function)
  must converge to this parabola.
  This is confirmed by plotting a couple of iterations.
  \begin{figure}[H]
    \centering
    \includegraphics[width=3\textwidth/4]{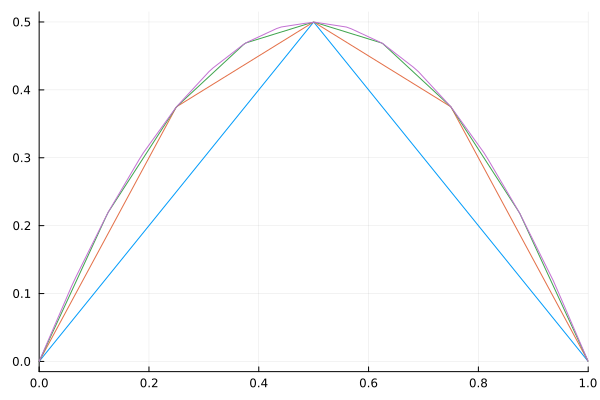}
    \vspace{-5pt}
    \caption{The first four (blue, orange, green, purple) iterates of the iRB operator
      when starting with the zero function. It is apparent that they converge to the parabola $x \mapsto 2x(1-x)$.}
  \end{figure}
\end{ex}

Similarly, we can realize the Takagi function (see \cite{takagi1973}) as a fixed point of the iRB operator (by multiplying $\tilde{s}_1$ and $\tilde{s}_2$ in the previous example by $2$) as it is the fixed point of an RB operator.

It should be noted that the extensions $l$, $q$ and $s$ in \cref{thm:iRB-gen-RB} are generally discontinuous
due to the discontinuity of $h$ at $\frac{1}{2}$.\\
We might also want to understand the RB operator as a special case of the iRB operator
that is induced by continuous $l$, $q$ and $s$.
For this purpose, we set $Z = B(X,F)$ and assume that we are given continuous $\tilde{l}_i$, $\tilde{q}_i$ and $\tilde{s}_i$
inducing an RB operator $\tilde{T}$ as in \eqref{eq:2}.

Then the idea is to replace the expression for $h = \1_{[\frac{1}{2},1]}$ given above by a sequence of continuous functions $h^{(k)}$
that converges pointwise to $h$;
for instance we can choose (for $k \ge 2$)
\[
  h^{(k)} \colon [0,1] \to [0,1],\ x \mapsto
  \begin{cases}
    0, & x \le \frac{1}{2} - \frac{1}{k};\\
    k\cdot (x + \frac{1}{k} - \frac{1}{2}), & \frac{1}{2} - \frac{1}{k} < x < \frac{1}{2};\\
    1, & x \ge \frac{1}{2}.
  \end{cases}
\]
Applying the above construction (with respect to $h^{(k)}$) to $\tilde{l}_i$, $\tilde{q}_i$, and $\tilde{s}_i$
(and again replacing $\tilde{q}_1$ by $2 \cdot \tilde{q}_1$ and similarly for $\tilde{q}_n$, $\tilde{s}_1$ and $\tilde{s}_n$)
yields a sequence of functions $l^{(k)}$, $q^{(k)}$, and $s^{(k)}$.
However, two subtleties arise in this context.
On the one hand, we need to assume that $X$ is convex in order for $l^{(k)}$ to be a function $[1,n] \times X \to X$
(instead of having codomain $E$), which is the setting we are interested in. (See \cref{sec:integral-RB-operator}.)
On the other hand, we need to ensure that the extensions $l^{(k)}$ have the property that for every $k \ge 2$,
$l_t^{(k)}$ is injective for almost all $t \in [1,n]$.
This is not automatically satisfied, as demonstrated by \cref{exa:0},
but it is guaranteed if the set
\[ \bc{t \in [0,1] : (1-t) \cdot \tilde{l}_j + t \cdot \tilde{l}_{j+1} \text{ not injective}} \]
is a (Lebesgue) null set for all $j \in\N_{n-1}$.

Hence, for every $k \ge 2$, we may construct a corresponding iRB operator $T^{(k)} \colon Z \to Z$.
By \cref{lem:extension-properties}, all $T^{(k)}$ are induced by continuous $l^{(k)}$, $q^{(k)}$ and $s^{(k)}$.

This construction allows us to approximate the RB operator, as demonstrated by the following theorem.

\begin{thm}[\textbf{iRB operator approximates the RB operator}]\label{thm:iRB-approx-RB}~\\
  Let $X\subset E$ be convex. For $i \in\N_n$, let $\tilde{l}_i \colon X \to X$ be injective and $\tilde{q}_i \colon X \to F$, $\tilde{s}_i \colon X \to \R$ be bounded.
  Furthermore, assume that all these functions are continuous.
  Denote their RB operator as defined in \eqref{eq:2} by $\tilde{T}$.
  If for all $j \in\N_{n-1}$, the set
  \[ \bc{t \in [0,1] : (1-t) \cdot \tilde{l}_j + t \cdot \tilde{l}_{j+1} \text{ not injective}} \]
  is a null set, then the sequence $T^{(k)} \colon B(X,F) \to B(X,F)$ of iRB operators constructed above
  is well-defined.
  It converges pointwise to the RB operator $\tilde{T}$.
\end{thm}
\begin{proof}
  We need to verify the pointwise convergence $T^{(k)} \to \tilde{T}$.
  By \cref{thm:iRB-gen-RB}, we can construct $l$, $q$ and $s$, such that the iRB operator $T$ induced
  by them is equal to the RB operator $\tilde{T}$.
  Therefore, it suffices to show that $T^{(k)} \to T$.
  
  Let $f \in B(X,F)$, $k \ge 2$ and set
  \[ 
  A(x) := \int_1^n (q_t \circ \inv{l_t})(x) \1_{X_t}(x) d t
  \]
  and
  \[ 
  B(x) := \int_1^n (s_t \circ \inv{l_t})(x) \cdot (f \circ \inv{l_t})(x) \1_{X_t}(x) d t, 
  \]
  so that $T(f) = A + B$. We employ a similar notation for $T^{(k)}$ to obtain $A^{(k)}$ and $B^{(k)}$ with $T^{(k)}(f) = A^{(k)} + B^{(k)}$. Then we have
  \[ \norm{T(f) - T^{(k)}(f)}_\infty \le \norm{A - A^{(k)}}_\infty + \norm{B - B^{(k)}}_\infty. \]
  By construction, $l^{(k)}_t$ agrees with $l_t$ for all $t \notin \N + (\frac{1}{2} - \frac{1}{k}, \frac{1}{2})$.
  The analogous statement is true for $q^{(k)}_t$ and $s^{(k)}_t$.
  Therefore, the first summand simplifies to
  \begin{equation}\label{ast}
   \sup_{x \in X} \sum_{i=1}^{n-1} \int_{i + \frac{1}{2} - \frac{1}{k}}^{i + \frac{1}{2}} \norm{(q_t \circ \inv{l_t})(x) \1_{X_t}(x) - (q^{(k)}_t \circ \inv{l^{(k)}_t})(x) \1_{X^{(k)}_t}(x)} d t. 
  \end{equation}
  By \cref{lem:extension-properties}, all $q_t$ and $q_t^{(k)}$ are bounded by $C \defi \max\limits_{i \in\N_n}\, \norm{\tilde{q}_i}_\infty$.
  Since this constant is independent of $k$, $t$ and $x$, \eqref{ast} is bounded above by
  \[ \sup_{x \in X} \sum_{i=1}^{n-1} \int_{i + \frac{1}{2} - \frac{1}{k}}^{i + \frac{1}{2}} 2 C d t = 2 (n-1) C \cdot \tfrac{1}{k}\: \stackrel{k \to \infty}{\longrightarrow}\: 0. \]
  The analogous argument applies to the second summand $\norm{B - B^{(k)}}_\infty$ which establishes the assertion.
\end{proof}

In general, this convergence is not uniform, as the following example shows.

\begin{ex}
  Consider $l$ from \cref{exa:1} but with $h = \1_{[\frac{1}{2},1]}$, that is, we have
  \[ l(t,x) = \tfrac{1}{2} \bb{x + \1_{[\frac{3}{2},2]}(t)}, \qquad \inv{l_t}(x) = 2x - \1_{[\frac{3}{2},2]}(t). \]
  Additionally, we choose $\tilde{q}_1 = \tilde{q}_2 = 0$ and $\tilde{s}_1 = \tilde{s}_2 = 1$.
  Then the extension of the $\tilde{q}_i$ is $0$ and the extension of the $\tilde{s}_i$ is $1$,
  so that the iRB operator is just
  \[ T(f)(x) = \int_1^2 (f \circ \inv{l_t})(x) \1_{X_t}(x) d t. \]
  We construct $l^{(k)}$, $q^{(k)}$ and $s^{(k)}$ as above and
  note that $q_t^{(k)} = 0$ and $s_t^{(k)} = 1$ for all $t \in [1,2]$ and $k \ge 2$.
  Moreover, just like in the previous proof, we observe that
  $l_t^{(k)} = l_t$ for $t \notin (\frac{3}{2} - \frac{1}{k}, \frac{3}{2})$ and $k \ge 2$.\\
  On the other hand, if $t \in (\frac{3}{2} - \frac{1}{k}, \frac{3}{2})$
  then $l_t = l_1$, implying that
  \[ \norm{T(f) - T^{(k)}(f)}_\infty
    \hspace{-5pt} = \hspace{-3pt} \sup_{x \in [0,1]} \abs{\int_{\frac{3}{2} - \frac{1}{k}}^{\frac{3}{2}} f(\inv{l_1}(x)) \1_{X_1}(x) - f(\inv{l_t^{(k)}}(x)) \1_{X_t^{(k)}}(x) d t}, \]
  where we have
  \begin{align*}
    X_1 &= l_1([0,1]) = \bs{0,\tfrac{1}{2}},\\
    X_t^{(k)} &= l_t^{(k)}([0,1]) = \bs{\tfrac{1}{2} k \bb{t + \tfrac{1}{k} - \tfrac{3}{2}}, \tfrac{1}{2} + \tfrac{1}{2} k \bb{t + \tfrac{1}{k} - \tfrac{3}{2}}}, 
  \end{align*}
  because by definition of $h^{(k)}$ and \eqref{eq:3}:
  \[ 
    l_t^{(k)}(x) = \tfrac{1}{2} \bb{x + k \cdot \bb{t + \tfrac{1}{k} - \tfrac{3}{2}}} , \quad \inv{l_t^{(k)}}(x) = 2x - k \cdot t + \tfrac{3}{2} k - 1.
  \]
  Note that $X_t^{(k)}$ contains $\frac{1}{2}$ for every $t \in [\frac{3}{2} - \frac{1}{k}, \frac{3}{2}]$, since
  it is a closed interval of length $\frac{1}{2}$ with infimum between $0$ and $\frac{1}{2}$.
  Choosing $f \in B([0,1],\R)$ such that $f(1) = 0$, it follows that
  \[ \norm{T(f) - T^{(k)}(f)}_\infty \ge \abs{\int_{\frac{3}{2} - \frac{1}{k}}^{\frac{3}{2}} f \bb{-k \cdot t + \tfrac{3}{2} k} d t} = \tfrac{1}{k} \abs{\int_0^1 f(u) d u}, \]
  where we used $x = \frac{1}{2}$ first and then made the substitution $u = -k \cdot t + \frac{3}{2} k$.
  Therefore, for an arbitrary $k \ge 2$, we may choose $f = k \cdot \1_{[0,1)} \in B([0,1], \R)$ which implies that
  $\norm{T(f) - T^{(k)}(f)}_\infty \ge 1$, showing that the convergence $T^{(k)} \to T$ is not uniform.
\end{ex}
\section{Continuous Solutions}\label{sec:continuous-sol}
We now examine the continuity of the unique fixed point that is guaranteed by \cref{thm:fixed-point-bounded}.
That is, we restrict the iRB operator $T \colon B(X,F) \to B(X,F)$ to the space of bounded, continuous functions $C^0_b(X,F)$.
More precisely, we assume that $f \in C^0_b(X,F)$ and that all assumptions of \cref{thm:fixed-point-bounded} are satisfied and investigate what conditions guarantee that $T(f) \in C^0_b(X,F)$.

To that end, we consider a convergent sequence $x_k \to x$ in $X$
and first focus on the affine part
\[ X \to F,\ x \mapsto \int_1^n (q_t \circ \inv{l_t})(x) \1_{X_t}(x) d t \]
of the iRB operator.
As $q \colon [1,n] \times X \to F$ is bounded,
the dominated convergence theorem (for Bochner integrals) implies that
\[ \lim_{k \to \infty} \int_1^n (q_t \circ \inv{l_t})(x_k) \1_{X_t}(x_k) d t = \int_1^n \lim_{k \to \infty} (q_t \circ \inv{l_t})(x_k) \1_{X_t}(x_k) d t, \]
assuming the limit $\lim_{k \to \infty} (q_t \circ \inv{l_t})(x_k) \1_{X_t}(x_k)$ exists
for almost all $t \in [1,n]$.
Since we want $T(f)$ to be continuous, we must ensure that this limit is equal to
$(q_t \circ \inv{l_t})(x) \1_{X_t}(x)$.\\
For this purpose, we introduce the following terminology.
Given an arbitrary family of functions $\bc{g_t \colon Y \to Z}_{t \in [1,n]}$, their \textit{set of discontinuity times} at $y \in Y$ is
\[ D_y(g) \defi \bc{t \in [1,n]: g_t \text{ discontinuous at $y$}} \subset [1,n]. \]

If we demand that for every $x \in X$, the set of discontinuity times $D_x(q_t \circ \inv{l_t})$
is a null set and that the same is true for $\1_{X_t} \colon X \to \R$, then as desired
\[ \lim_{k \to \infty} (q_t \circ \inv{l_t})(x_k) \1_{X_t}(x_k) = (q_t \circ \inv{l_t})(x) \1_{X_t}(x) \]
for almost all $t \in [1,n]$.
By the following lemma, the latter assumption can be simplified.

\begin{lem}
  For $t \in [1,n]$, we have
  \[ x \in \partial X_t \ \iff\ \1_{X_t} \colon X \to \R \text{ discontinuous at $x$}. \]
  Here, the boundary operator $\partial$ is understood with respect to $X$.
\end{lem}
\begin{proof}
  Assume that $x \in \partial X_t$ and consider the case that $x \notin X_t$.
  As $x \in \overline{X_t}$ (the closure of $X_t$) there exists a sequence $x_k \in X_t$ that converges to $x$ such that $\1_{X_t}$ is discontinuous at $x$.
  If instead $x \in X_t$, then there exists a sequence $x_k \notin X_t$ that converges to $x$;
  for if not, then there must be some open neighborhood of $x$ in $X$ that is contained in $X_t$,
  implying that $x$ is in the interior of $X_t$.\\
  On the other hand, suppose that $x \notin \partial X_t$.
  If $x \notin \overline{X_t}$, then any sequence converging to $x$ must eventually lie in the complement of $X_t$,
  so that $\1_{X_t}$ is continuous at $x$.
  Similarly, if $x$ is in the interior of $X_t$, then any sequence converging to $x$ must eventually lie in
  that interior, showing that $\1_{X_t}$ is continuous at $x$.
\end{proof}

Therefore, for any $x \in X$, we have
\[ \bc{t \in [1,n]: \1_{X_t} \colon X \to \R \text{ discontinuous at $x$}} = \bc{t \in [1,n] : x \in \partial X_t} \]
and we demand that this latter set is a null set.\\
An analogous argument for the linear part
\[ X \to F,\ x \mapsto \int_1^n (s_t \circ \inv{l_t})(x) \cdot (f \circ \inv{l_t})(x) \1_{X_t}(x) d t \]
of the iRB operator shows that $T(f)$ is continuous
if we additionally require $D_x(\inv{l_t})$ and $D_x(s_t \circ \inv{l_t})$ to be null sets for all $x \in X$.
Note that this is in particular the case when these functions are continuous for almost all $t \in [1,n]$.

Under these assumptions, the iRB operator $T \colon B(X,F) \to B(X,F)$ restricts to $C^0_b(X,F)\to C^0_b(X,F)$ and by applying the Banach fixed point theorem, we obtain a bounded continuous fixed point. By uniqueness, it must also be the fixed point of $T \colon B(X,F) \to B(X,F)$.
This establishes the following theorem.

\begin{thm}\label{thm:fixed-point-cont}
  Within the setting of \cref{thm:fixed-point-bounded}, 
  assume further that for all $x \in X$, the discontinuity times $D_x(\inv{l_t})$, $D_x(q_t \circ \inv{l_t})$ and $D_x(s_t \circ \inv{l_t})$ are null sets
  (which in particular is the case if these functions are continuous for almost all $t \in [1,n]$)
  and that the set $\bc{t \in [1,n] \colon x \in \partial X_t}$ is also a null set for all $x \in X$.\\
  Then, the unique fixed point of the iRB operator $T \colon B(X,F) \to B(X,F)$ is continuous.
\end{thm}

\begin{ex}
  Returning to \cref{exa:1}, we first observe that since $X_t = \bs{\frac{1}{2} t - \frac{1}{2}, \frac{1}{2} t}$
  for $t \in [1,2]$, the set
  $\bc{t \in [1,2] : x \in \partial X_t}$ contains at most two points and is thus a null set for all $x \in [0,1]$.
  Let $x_k \to \frac{1}{2}$ be a sequence converging to $\frac{1}{2}$ from below (so $x_k < \frac{1}{2}$ for all $k \in \N$). Then
  \[ \inv{l_t}(x_k) = 2x_k - t + 1 < 2 - t, \]
  implying $q_t \bb{\inv{l_t}(x_k)} = 0$, even though $q_t \bb{\inv{l_t}\bb{\frac{1}{2}}} = 1$,
  which shows that $q_t \circ \inv{l_t}$ is discontinuous at $\frac{1}{2}$ for all $t \in [1,2]$.
  Therefore, the assumptions of \cref{thm:fixed-point-cont} are not fulfilled.
  Instead, the affine part of the iRB operator satisfies
  \[ \lim_{k \to \infty} \int_1^2 (q_t \circ \inv{l_t})(x_k) \1_{X_t}(x_k) d t = \int_1^2 \bb{\lim_{k \to \infty} 0} d t = 0, \]
  and thus does not converge to
  \[ \int_1^2 q_t \bb{\inv{l_t}\bb{\frac{1}{2}}} d t = \int_1^2 1 d t = 1. \]
  If $f$ is continuous, then the linear part of the iRB operator is continuous, as $s_t$ and $\inv{l_t}$ are continuous for all $t \in [1,2]$.
  Consequently, $T(f)$ is discontinuous and in particular the unique fixed point must be discontinuous as well.
  This is confirmed by \cref{fig:exa-disc}.
\end{ex}

\begin{ex}
  In \cref{exa:2}, it is clear that $\inv{l_t}$, $q$ and $s$ are continuous, which together with the observation regarding $\bc{t \in [1,2] : x \in \partial X_t}$
  from the previous example implies that the solution must be continuous.
  This is also apparent from \cref{fig:exa-cont}.
\end{ex}
%\R^m
\section{Solutions in $L^p$}\label{sec:solutions-in-lp}
Let us now consider $X \subset E \defi \R^d$ and $F \defi \R^m$ with the usual Euclidean norm and the $L^p$ space $Z = L^p(X,\R^m)$ (for $p \in [1,\infty)$), which is equipped with the $p$-norm
\[ \norm{f} \defi \bb{\int_X \norm{f(x)}^p dx}^{\frac{1}{p}}. \]
As before, we first need to ensure that $T$ maps from $Z$ to $Z$,
meaning that $T(f) \in L^p(X,\R^m)$ whenever $f \in L^p(X,\R^m)$.
To this end, we calculate (for arbitrary $x \in X$)
\begin{align*}
\left\|\int_1^n (q_t \circ \inv{l_t})(x)\right. & \left. \1_{X_t}(x) d t + \int_1^n (s_t \circ \inv{l_t})(x) \cdot (f \circ \inv{l_t})(x) \1_{X_t}(x) d t\right\|^p\\
  &\le 2^{p-1} \left[\, \n{\int_1^n (q_t \circ \inv{l_t})(x) \1_{X_t}(x) d t}^p\right.\\
  &\quad + \left.\n{\int_1^n (s_t \circ \inv{l_t})(x) \cdot (f \circ \inv{l_t})(x) \1_{X_t}(x) d t}^p\,\right]\\
  &\le 2^{p-1} \left[\left(\int_1^n \n{(q_t \circ \inv{l_t})(x)} \1_{X_t}(x) d t\right)^{\hspace{-2pt} p}\right.\\
  &\quad + \left.\left(\int_1^n \abs{(s_t \circ \inv{l_t})(x)} \cdot \n{(f \circ \inv{l_t})(x)} \1_{X_t}(x) d t\right)^{\hspace{-2pt} p}\right],
\end{align*}
where we first utilized Jensen's inequality and then the monotonicity of $x \mapsto x^p$.
Another application of Jensen's inequality reveals that
\[ \bb{\int_1^n \norm{(q_t \circ \inv{l_t})(x)} \1_{X_t}(x) d t}^{\hspace{-2pt} p} \le (n-1)^{p-1} \int_1^n \norm{(q_t \circ \inv{l_t})(x)}^p \1_{X_t}(x) d t. \]
The second summand can be estimated analogously. Putting these inequalities together, we obtain
\begin{align*}
  \norm{T(f)}^p
  &= \int_X \left\| \int_1^n (q_t \circ \inv{l_t})(x) \1_{X_t}(x) d t \right.\\
  &\quad + \left. \int_1^n (s_t \circ \inv{l_t})(x) \cdot (f \circ \inv{l_t})(x) \1_{X_t}(x) d t \right\|^p dx\\
  &\le (2(n-1))^{p-1} \int_X \bigg( \int_1^n \norm{(q_t \circ \inv{l_t})(x)}^p \1_{X_t}(x) dt\\
  &\quad + \int_1^n \abs{(s_t \circ \inv{l_t})(x)}^p \cdot \norm{(f \circ \inv{l_t})(x)}^p \1_{X_t}(x) dt \bigg) dx.
\end{align*}
By Tonelli's theorem, this is equal to
\begin{align*}
  (2(n-1))^{p-1} \int_1^n & \left( \int_{X_t} \norm{(q_t \circ \inv{l_t})(x)}^p dx \right.\\
  &\left. + \int_{X_t} \abs{(s_t \circ \inv{l_t})(x)}^p \cdot \norm{(f \circ \inv{l_t})(x)}^p dx \right) dt.
\end{align*}
Assuming that $l_t \colon X \to X_t$ is a diffeomorphism for almost all $t \in [1,n]$,
we may apply the Jacobi transformation formula to simplify this expression to
\begin{align*}
  (2(n-1))^{p-1} \int_1^n & \left( \int_X \norm{q_t(z)}^p \abs{\det(D l_t(z))} dz \right.\\
  &\left. + \int_X \abs{s_t(z)}^p \cdot \norm{f(z)}^p \abs{\det(D l_t(z))} dz \right) dt.
\end{align*}
Moreover, assuming that $X$ is bounded and that $\sup\limits_{t \in [1,n], z \in X} \abs{\det(D l_t(z))}$ is finite,
it suffices to require that 
\[
q \in L^p([1,n] \times X, \R^m)\quad\text{and}\quad s \in L^\infty([1,n] \times X, \R^m),
\]
in order to guarantee that the above expression is finite.

This leads us to the following theorem, which is similar to \cref{thm:fixed-point-bounded}.

\begin{thm}\label{thm:fixed-point-lp}
  Let $X \subset \R^d$ be a bounded subset, $q \in L^p([1,n] \times X, \R^m)$ and $l_t \colon X \to X_t$ a diffeomorphism for almost all $t \in [1,n]$. Furthermore assume that the suprema
  \[ S \defi \sup_{t \in [1,n], x \in X} \abs{s(t,x)}, \qquad L \defi \sup_{t \in [1,n], x \in X} \abs{\det(D l_t(x))} \]
  are finite and satisfy
  \[ (n-1) \cdot S \cdot L^{\frac{1}{p}} < 1. \]
  Then, there exists a unique fixed point of the iRB operator
  $T \colon L^p(X,\R^m) \to L^p(X,\R^m)$.
\end{thm}
\begin{proof}
  We need to verify that $T$ is contraction. By definition, we have
  \begin{align*}
    \norm{T(f) - T(g)}^p
    &= \int_X \norm{\int_1^n (s_t \circ \inv{l_t})(x) \cdot ((f-g) \circ \inv{l_t})(x) \1_{X_t}(x) dt}^p dx
  \end{align*}
  and arguments similar to those above yield the upper bound
  \begin{align*}
    &\quad\ (n-1)^{p-1} \int_1^n \int_X \abs{s_t(z)}^p \cdot \norm{(f-g)(z)}^p \abs{\det(D l_t(z))} dz dt\\
    &\le (n-1)^p S^p L \norm{f-g}^p,
  \end{align*}
  so $T$ is a contraction whenever $(n-1) \cdot S \cdot L^{\frac{1}{p}} < 1$.
\end{proof}
\noindent
We demonstrate the theorem by an example.
\begin{ex}
  Choose $l \colon [1,2] \times [0,1) \to [0,1)$ as in \cref{exa:3}, that is,
  \[ l(t,x) = \tfrac{1}{2} x + \1_{[\frac{3}{2},2]}(t) \tfrac{1}{2}. \]
  Furthermore, let (with $q(t,0) \in \R$ arbitrary)
  \[ q(t,x) = \frac{1}{\sqrt{x}} \in L^1([1,2] \times [0,1), \R), \qquad s(t,x) = \tfrac{3}{2}. \]
  Then $S = \frac{3}{2}$ and $L = \frac{1}{2}$, so that $(2-1) \cdot \frac{3}{2} \cdot \frac{1}{2} = \frac{3}{4} < 1$ and
  thus \cref{thm:fixed-point-lp} implies the existence of a unique fixed point of the iRB operator
  $T \colon L^1([0,1),\R) \to L^1([0,1),\R)$.
    \begin{figure}[H]
    \centering
    \includegraphics[width=3\textwidth/4]{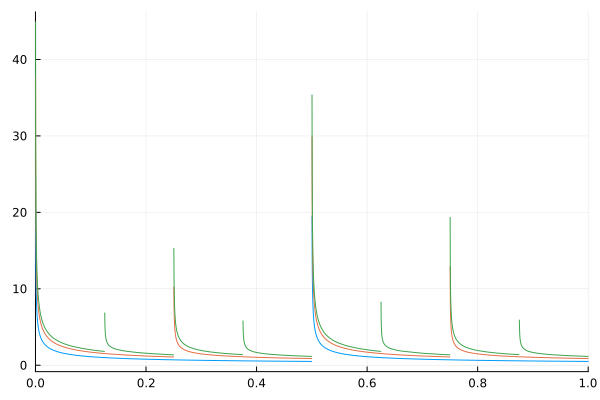}
    \vspace{-5pt}
    \caption{The first three (blue, orange, green) iterates of the iRB operator
      when starting with the zero function.
      The first iterate has two ``spikes'' at  $x=0$ and $x = \frac{1}{2}$ and their number doubles with each iteration.}
  \end{figure}
\end{ex}

% ------------------------------------------------------------------------

% ------------------------------------------------------------------------

\begin{thebibliography}{1}
\bibitem{barnsley1986} M. F. Barnsley, {Fractal Functions and Interpolation.} \textit{Constr. Approx.} \textbf{2}(1) (1986), 303--329.
%
\bibitem{BHVV} M.F. Barnsley, B. Harding, C. Vince, P. Viswanathan, { Approximation of Rough Functions.} \textit{J. Approx. Th.} \textbf{209} (2016), 23--43. 
%
\bibitem{bhm} M. F. Barnsley, M. Hegland, P. Massopust, Numerics and Fractals. \textit{Bull. Inst. Math. Acad. Sin. (N.S.)} \textbf{9(3)} (2014), 389--430.
%
%\bibitem{BH} Barnsley, M.F., Hurd, L.P. \textit{Fractal Image Compression}. AK Peters Ltd., Wellesly, MA, 1993.
%
\bibitem{bedford} T. Bedford, M. Dekking, M. Keane, Fractal Image Coding Techniques and Contraction Operators. \textit{Delft University of Technology Report} (1992), {92--93}.
%
\bibitem{DLM} N. Dira, D. Levin, P. Massopust, Attractors of Trees of Maps and of Sequences of Maps Between Spaces and Applications to Dubdivision. \textit{J. Fixed Point Theory Appl.} \textbf{22}(14) (2020), 1--24.
%
\bibitem{dubuc2} S. Dubuc, Interpolation Fractale. In \textit{Fractal Geomety and Analysis}, J. B\'elais and S. Dubuc, eds., Kluwer Academic Publishers, Dordrecht, The Netherlands, 1989.

\bibitem{H} J. E. Hutchinson, Fractals and Self-similarity. \textit{Indiana Univ. Math. J.} \textbf{30} (1981), 713--747.

\bibitem{KCM} D. Kumar, A. K. B. Chand, P. Massopust, Approximation by Quantum Meyer-K\"onig-Zeller Fractal Functions. \textit{fractal and fractional} \textbf{6}(12) (2022), 704.

\bibitem{LDV} D. Levin, N. Dyn, P. Viswanathan, Non-stationary Versions of Fixed-Point Theory, with Applications to Fractals and Subdivision. \textit{J. Fixed Point Theory Appl.}  \textbf{21} (2019), 1--25.

%\bibitem{PRM} Massopust, P.R. Fractal functions and their applications. \textit{Chaos, Solitons and Fractals}, \textbf{1997}, \textit{8(2)}, 171--190.
% 
%\bibitem {massopust} Massopust, P.R. \textit{Interpolation and Approximation with Splines and Fractals,} Oxford University Press: Oxford, USA, 2010.

\bibitem{massopust1} Massopust, P.R. \textit{Fractal Functions, Fractal Surfaces, and Wavelets}, 2nd ed., Academic Press: San Diego, USA, 2016.

%\bibitem{m2} Massopust, P.R. Local fractal functions and function spaces. \textit{Springer Proceedings in Mathematics \& Statistics: Fractals, Wavelets and their Applications} \textbf{2014}, Vol. 92, 245--270.
%
%\bibitem{m3} Massopust, P.R. Local Fractal Functions in Besov and Triebel-Lizorkin Spaces. \textit{J. Math. Anal. Appl.} \textbf{2016}, \textit{436}, 393 -- 407. 
%
%\bibitem{m7} Massopust, P.R. Local fractal interpolation on unbounded domains. \textit{Proc. Edinburgh Math. Soc.} \textbf{2016}, \textit{61}, 151--167.
%
%\bibitem{m4} Massopust, P.R. Non-Stationary Fractal Interpolation. \textit{Mathematics} \textbf{2019}, \textit{7(8)}, 1 -- 14.

\bibitem{m5} Massopust, P.R. Hypercomplex Iterated Function Systems. In: \textit{Current Trends in Analysis, its Applications and Computation.}
Proceedings of the 12th ISAAC Congress, Aveiro, Portugal, Cereijeras, P.; Reissig, M.; Sabadini, I.; Toft, J. (eds.), Birkh\"auser (2022), 589--598.

\bibitem{massopust2022} P. Massopust, \textit{Fractal Interpolation: {{From}} Global to Local, to Nonstationary and Quaternionic.} In: Frontiers of {{Fractal Analysis. Recent Advances}} and {{Challenges}}, CRC Press (2022), 24--48.
Birkh\"auser, 2000.

\bibitem{N} Navascu\'es, M.A. Fractal Polynomial Interpolation. \textit{Z. Anal. Anwendungen} \textbf{24}(2) (2005), 401--418.

%\bibitem{R} Rolewicz, S. \textit{Metric Linear Spaces}, Kluwer Academic: Warsaw, Poland, 1985.
%
%\bibitem{rudin} Rudin, W. {\em Functional Analysis}, McGraw--Hill: New York, 1991.
%
\bibitem{SB} C. Serpa, J. Buescu, Constructive Solutions for Systems of Iterative Functional Equations. \textit{Constructive Approx.} \textbf{45}(2) (2017), 273--299.

\bibitem{SB1} C. Serpa, J. Buescu, Compatibility conditions for systems of iterative functional equations with non-trivial contact sets. \textit{Results Math.} \textbf{2} (2021), 1--19.

\bibitem{takagi1973} T. Takagi, {A Simple Example of a Continuous Function Without Derivative.} \textit{Proc. Soc. Japan} \textbf{1} (1903), 176--177.
\end{thebibliography}
\end{document}